\documentclass[a4paper,11pt,reqno]{amsart}
\usepackage{amsmath,amssymb,amsthm,mathtools,bbm,fullpage,cleveref,graphicx}

\usepackage{xcolor} 

\newtheorem{thm}{Theorem}[section]
\crefname{thm}{Theorem}{Theorems}
\newtheorem{cor}[thm]{Corollary}
\crefname{cor}{Corollary}{Corollaries}
\newtheorem{lem}[thm]{Lemma}
\crefname{lem}{Lemma}{Lemmas}
\newtheorem{clm}[thm]{Claim}

\newcommand\Z{\mathbb{Z}}
\newcommand\N{\mathbb{N}}
\newcommand\R{\mathbb{R}}
\newcommand\T{\mathbb{T}}
\newcommand\E{\mathbb{E}}
\newcommand\Prb{\mathbb{P}}
\newcommand\eps{\varepsilon}
\newcommand\sdiff{\mathbin{\triangle}}
\newcommand\id{\mathbbm{1}}
\newcommand\medcup{\mathbin{\scalebox{0.95}{\ensuremath{\bigcup}}}} 
\newcommand\medcap{\mathbin{\scalebox{0.95}{\ensuremath{\bigcap}}}} 

\title{Random Translates in Minkowski Sums}
\author{Paul Balister \and B\'ela Bollob\'as \and Imre Leader \and Marius Tiba}

\address{Mathematical Institute, University of Oxford, Oxford OX2\thinspace6GG, UK.}
\email{paul.balister@maths.ox.ac.uk}

\address{Department of Pure Mathematics and Mathematical Statistics,
Wilberforce Road, Cambridge, CB3\thinspace0WA, UK, and Department of Mathematical Sciences,
University of Memphis, Memphis, TN\thinspace38152, USA.}
\email{b.bollobas@dpmms.cam.ac.uk}

\address{Department of Pure Mathematics and Mathematical Statistics,
Wilberforce Road, Cambridge, CB3\thinspace0WA, UK.}
\email{i.leader@dpmms.cam.ac.uk}

\address{Mathematical Institute, University of Oxford, Oxford OX2\thinspace6GG, UK.}
\email{marius.tiba@maths.ox.ac.uk}

\thanks{PB was partially supported by EPSRC grant EP/W015404/1.
BB was partially supported by NSF grant DMS-1855745.}

\date{\today}

\begin{document}

\maketitle

\begin{abstract}
Suppose that $A$ and $B$ are sets in~$\R^d$, and we form the sumset of $A$ with $n$
random points of~$B$. Given the volumes of $A$ and~$B$, how should we choose them to
minimize the expected volume of this sumset? Our aim in this paper is to show that we
should take $A$ and $B$ to be Euclidean balls. We also consider the analogous question
in the torus~$\T^d$, and we show that in this case the optimal choices of $A$ and $B$
are bands, in other words, sets of the form $[x,y]\times\T^{d-1}$. We also give stability versions of our results. 
\end{abstract}

\section{Introduction}

Let $A$ and $B$ be sets of finite positive measure in~$\R^d$. The fundamental 
Brunn--Minkowski inequality (see for example~\cite{Bar,Gard}) asserts that the volume 
$|A+B|$ of the sumset
$A+B$ is at least $(|A|^{1/d}+|B|^{1/d})^d$.
Note that we may obtain equality for any convex
body~$A$ by setting $B$ to be a homothetic copy of~$A$. But suppose we try to
`approximate' $B$ by a finite set -- in the sense that,  for a fixed value
of~$n$, we form the sumset $A+S$, where $S$ consists of $n$ points chosen independently
and uniformly at random from~$B$. Given their volumes, how should we choose $A$ and~$B$ 
so as to minimize the expected size of $A+S$?

Note that now there is no reason at all to believe that this minimum may be attained
using any given (convex) choice of $A$ by a suitably related choice of~$B$, unlike
for the Brunn--Minkowski inequality itself. Indeed, it is natural to guess, based for example upon isoperimetric notions, that Euclidean
balls might actually be best. And this is what we prove.

\begin{thm}\label{main_sphere}
 Let $d,n\in\N$ and let $A$ and\/ $B$ be sets of finite positive measure in~$\R^d$.
 Let $A'$ and\/ $B'$ be balls in $\R^d$ with the same volumes as
 $A$ and\/~$B$ respectively. Then for $b_1,\dots,b_n$ and\/ $b'_1,\dots,b'_n$ chosen
 independently and uniformly at random from $B$ and\/~$B'$, respectively, we have 
 \[
  \E(|\medcup_i(A+b_i)|) \ge \E(|\medcup_i(A'+b'_i)|).
 \]
\end{thm}

From this result, we may obtain estimates on how many random points of $B$ we need to
take to ensure that the expected volume is close to the true lower bound on
$|A+B|$ given by the Brunn--Minkowski inequality. Indeed, \Cref{main_sphere} has the following consequence.

\begin{cor}\label{estimate_sphere}
 Let $d,n\in\N$ and let $A$ and\/ $B$ be sets of finite positive measure in~$\R^d$.
 Then for $b_1,\dots,b_n$ chosen independently and uniformly at random from $B$ we have 
 \[
  \E(|\medcup_i(A+b_i)|) \ge (|A|^{1/d}+|B|^{1/d})^d \cdot\big(1-(c+o(1))n^{-2/(d+1)}\big)
  \qquad\text{as }n\to\infty,
 \]
 where $c$ is a constant that depends only on the dimension $d$ and the ratio
 of the volumes of\/ $A$ and\/~$B$.
\end{cor}

It is also very natural to ask similar questions in the torus~$\T^d$, where $\T$ denotes
the circle $\R/\Z$. This time the analogue of the Brunn--Minkowski inequality would state that bands,
meaning sets of the form $[x,y]\times\T^{d-1}$, are best. More precisely, if $A$ and $B$
are (measurable) sets in $\T^d$ then $|A+B| \ge |A|+|B|$ (if the sum of the two volumes is at most
$|\T^d|=1$, of course -- so strictly one should write $\min\{|A|+|B|,1\}$). This is a result of 
Macbeath~\cite{macbeath} -- it can also be deduced immediately from Kneser's theorem~\cite{Knes}
in the group~$\Z_p^d$. We show that indeed bands are always best to minimize the expected volume
of $A+S$, where $S$ consists of $n$ points chosen uniformly at random from~$B$.

\begin{thm}\label{main_torus}
 Let $d,n\in\N$ and let $A$ and\/ $B$ be sets of positive measure in~$\T^d$. Let $A'$
 and\/ $B'$ be bands in $\T^d$ \textup{(}oriented in the same direction\textup{)} with the
 same volumes as $A$ and\/ $B$ respectively, i.e., sets of the form $I\times\T^{d-1}$,
 where $I$ is a closed interval in~$\T$. Then for $b_1,\dots,b_n$ and\/ $b'_1,\dots,b'_n$
 chosen independently and uniformly at random from $B$ and $B'$ respectively, we have 
 \[
  \E(|\medcup_i(A+b_i)|) \ge \E(|\medcup_i(A'+b'_i)|).
 \]
\end{thm}
To keep the notation simple we use the same notation $|\cdot|$ for Lebesgue measure in the
torus as in Euclidean space --- there is no place where this is ambiguous.

Again, \Cref{main_torus} allows us to estimate how many points we need to take to ensure that the sum
of $A$ and $n$ points of $B$ is, on average, of volume at least $\min\{|A|+|B|,1\}-\eps$.

\begin{cor}\label{estimate_intervals}
 Let $n\ge 1$ and let $A$ and\/ $B$ be sets of positive measure in $\T^d$. Then for
 $b_1,\dots,b_n$ chosen independently and uniformly at random from $B$ we have 
 \[
  \E(|\medcup_i(A+b_i)|) \ge |A|+\tfrac{n-1}{n+1} |B|
   -\tfrac{n-1}{n+1} |B|\big(1-\tfrac{|A|}{|B|}\big)^{n+1}\id_{|A|<|B|},
 \]
 when $|A|+|B|\le 1$, and
 \[
  \E(|\medcup_i(A+b_i)|) \ge 
   1-\big(\tfrac{1-|A|}{|B|}\big)^n\big(|B|-\tfrac{n-1}{n+1}(1-|A|)\big)
   -\tfrac{n-1}{n+1} |B|\big(1-\tfrac{|A|}{|B|}\big)^{n+1}\id_{|A|<|B|},
 \]
 when $|A|+|B|>1$, where $\id_{|A|<|B|}$ is $1$ if\/ $|A|<|B|$ and\/ $0$ otherwise.
\end{cor}

Our approach to our main result in Euclidean space, \Cref{main_sphere}, proceeds by first proving a discrete
version of the result in one dimension, so in the group~$\Z_p$. Thus suppose we have sets
$A$ and $B$ in $\Z_p$ of given size, and we want to minimize the expected size of the sum
of $A$ with $n$ random points of~$B$. Then it is best to take $A$ and $B$ to be intervals.
Although this might seem very natural, it is actually quite non-trivial to prove. Indeed,
this result, or rather a generalization where we choose the $n$ points not all from the
same set but from different sets, is at the heart of our approach, and is where most of the
difficulty lies. Once we have proved this result, the general result will follow by some
reasonably standard arguments of discrete approximation and Steiner symmetrization. 

\begin{thm}\label{main_tool}
 Let $n,p\in\N$ with $p$ prime and let $A_1,\dots,A_n$ and\/ $B_1,\dots,B_n$ be finite non-empty 
 subsets of\/~$\Z_p$. Let $A'_1,\dots,A'_n$ and\/ $B'_1,\dots,B'_n$ be intervals of the
 same sizes as the sets $A_1,\dots,A_n$ and\/ $B_1,\dots B_n$ respectively, with each
 $A'_i+B'_i$ centred\footnote{In the case when $A'_i+B'_i=\Z_p$ we mean that the sum of the
 centres of $A'_i$ and $B'_i$ is $0$ or $1/2$.} at\/ $0$ or $1/2$ \textup{(}depending on parity\textup{)}.
 Then for $b_i$, $b'_i$ chosen independently and uniformly at random from $B_i$, $B'_i$,
 respectively, we have
 \[
  \E(|\medcup_i (A_i+b_i)|) \ge \E(|\medcup_i (A'_i+b'_i)|).
 \]
\end{thm}

It might seem strange to be using the cyclic group, instead of the group $\Z$ of integers,
to approximate the continuous space~$\R$, but in fact this is vital to our approach.
Roughly speaking, this is because in $\Z_p$ the complement of a finite set is still finite
--- this will allow us to transform unions into intersections. As we shall see, the size of
the \emph{intersection} of some translates of a set $A$ is much easier to deal with than the size
of a \emph{union} of translates. 

In the torus, our starting point is again the above discrete result. We pass to a continuous version,
and then we try to symmetrize on the torus. However, the torus is not isotropic, so one does not
have the power of `symmetrizing in all possible directions' that one has in~$\R^d$.
Nevertheless, it turns out that there are just enough symmetrizations that we are allowed to
perform that we can solve the problem in two dimensions: essentially, we can symmetrize in
any `direction' given by a member of $SL_2(\Z)$. This is where most of the work comes: armed with the 
two-dimensional version
of the theorem, it is not hard to pass to the general case.

We also consider stability versions of these results. In other words, if our sets $A$ and $B$ give a
value for the expectation that is `close' to the lower bound, how close must $A$ and $B$ be to the
extremal sets for the inequality? We are able to prove stability results for each our two main results that are
sharp, in the sense that the
dependences among the notions of `close' are best possible. These are 
\Cref{stab_Rd} and \Cref{stab} below:
the former as the stability version of \Cref{main_sphere} and the latter for \Cref{main_torus}.   
Interestingly, the proofs of these
stability statements do not rely upon the actual inequalities themselves, so that these may be viewed
as also providing alternative proofs of the results (\Cref{main_sphere} and \Cref{main_torus}). However, these
stability proofs do rely on several deeper recent tools.

The plan of the paper is as follows. In Section~\ref{tool} we prove our main tool, 
\Cref{main_tool}. In Section~\ref{continuous} we translate this result into the
continuous setting of the 1-dimensional torus. In Section~\ref{torus} we prove \Cref{main_torus}, our 
result on the torus, and in the next section we 
finish the proof of our result in~$\R^d$, \Cref{main_sphere}.
Then Section~\ref{calculations} contains the calculations needed for the actual estimates
that follow from our results: these are routine, although somewhat fiddly.
In Section~\ref{stability} we prove our sharp stability result for \Cref{main_torus}, and
then in Section~\ref{stability_sphere} we prove a corresponding stability result for
\Cref{main_sphere}. We conclude with some open problems.

Our notation is standard. However, we would like to highlight one of the conventions that we use.
When we have a measurable set $A$, say in~$\R^d$, we often write as though all of its various
sections in a given direction were measurable sets in $\R^{d-1}$. Of course, this need not
be the case, but as it only fails for a set of sections that has measure zero, this
will never affect any of our arguments. 

Finally, we mention one more useful convention. Often we have intervals, in $\Z_p$ or $\T$,
centred at various points. We always allow the whole space $\Z_p$ or $\T$ as an interval,
and we consider it to be `centred' at any point we choose.

\section{Proof of \Cref{main_tool}}\label{tool}

\begin{proof}[Proof of \Cref{main_tool}]
The case $p = 2$ is easy to check, so we may assume that $p$ is an odd prime.
We need to show that for $b_i$, $b'_i$ chosen independently and uniformly at random
from $B_i$, $B'_i$, respectively, we have
\[
 \E(|\medcup_i (A_i+b_i)|) \ge \E(|\medcup_i (A'_i+b'_i)|),
\]
which after taking complements inside $\Z_p$ is equivalent to
\[
 \E(|\medcap_i (A_i^c+b_i)|) \le \E(|\medcap_i (A^{\prime c}_i+b'_i)|).
\]
After scaling and using $|B_i|=|B'_i|$, we can rewrite the last inequality as 
\[
 \sum_{b_1\in B_1,\dots,b_n\in B_n} |\medcap_i (A_i^c+b_i)|
 \le \sum_{b'_1\in B'_1,\dots,b'_n\in B'_n} |\medcap_i (A^{\prime c}_i+b'_i)|,
\]
which is equivalent to
\[
 \sum_{x\in\Z_p}\prod_i |\{(d_i,b_i)\in A_i^c\times B_i : d_i+b_i=x\}|
 \le \sum_{x\in\Z_p}\prod_i
 |\{(d^{\prime}_i,b'_i)\in A^{\prime c}_i\times B'_i : d^{\prime}_i+b'_i=x\}|,
\]
or in other words
\[
 \sum_{x\in\Z_p}\prod_i\id_{A_i^c}*\id_{B_i}(x)
 \le \sum_{x\in\Z_p}\prod_i\id_{A^{\prime c}_i}*\id_{B'_i}(x),
\]
where $\id_X$ represents the characteristic function of $X\subseteq\Z_p$ and
$*$ denotes convolution. To simplify the notation, set
\[
 D_i=A_i^c\qquad\text{and}\qquad
 D'_i=A^{\prime c}_i + \tfrac{p-1}{2},
\]
so as to centre $D'_i+B'_i$ at $0$ or $-1/2$
(depending on parity). Note that $D'_i$ is an interval with the same size as~$D_i$.
With this notation, the objective is to show that
\begin{equation}\label{eq_equiv}
 \sum_{x\in\Z_p}\prod_i\id_{D_i}*\id_{B_i}(x)
 \le \sum_{x\in\Z_p}\prod_i\id_{D'_i}*\id_{B'_i}(x).
\end{equation}

In order to show the above inequality we shall exploit just the following properties
of these functions.
\begin{clm}\label{clm}
 For any $1\le i\le n$ we have that
 \begin{enumerate}
  \item for $j\ge0$, $\sum_{x\in\Z_p} \max(0,\id_{D_i}*\id_{B_i}(x)-j)
   \le \sum_{x\in\Z_p} \max(0,\id_{D'_i}*\id_{B'_i}(x)-j)$, and
  \item with $(x_0,x_1,\dots,x_{p-1})=(0,-1,1,-2,2,\dots,-\frac{p-1}{2},\frac{p-1}{2})$,
   the sequence $\id_{D'_i}*\id_{B'_i}(x_t)$ is decreasing.
 \end{enumerate}
\end{clm}
\begin{proof}[Proof of \Cref{clm}]
For the first part, recall Theorem~2 in~\cite{pollard}. This states that for fixed
$1\le i\le n$  and $j\ge 0$ we have
\begin{equation}\label{e:T2p}
 \sum_{x\in\Z_p} \min(\id_{D_i}*\id_{B_i}(x),j)
 \ge \sum_{x\in\Z_p} \min(\id_{D'_i}*\id_{B'_i}(x),j).
\end{equation}
As $|D_i|=|D'_i|$ and $|B_i|=|B'_i|$ we note that
\[
 \sum_{x \in \Z_p} \id_{D_i}*\id_{B_i}(x) = \sum_{x \in \Z_p} \id_{D_i'}*\id_{B_i'}(x).
\]
Thus \eqref{e:T2p} is equivalent to
\[
 \sum_{x\in\Z_p} \big(\id_{D_i}*\id_{B_i}(x)-\min(\id_{D_i}*\id_{B_i}(x),j))\big)
 \le \sum_{x\in\Z_p} \big(\id_{D'_i}*\id_{B'_i}(x)-\min(\id_{D'_i}*\id_{B'_i}(x),j)\big),
\]
i.e.,
\[
 \sum_{x\in\Z_p} \max(0,\id_{D_i}*\id_{B_i}(x)-j)
 \le \sum_{x\in\Z_p} \max(0,\id_{D'_i}*\id_{B'_i}(x)-j),
\]
as desired. 

The second part is just a simple computational check. However, we point out that it is
here where we use the fact that the interval $D'_i+B'_i$ is almost symmetric about the origin.
\end{proof}

We say two functions $h_1,h_2\colon\Z_p\to\N$ have the \emph{same profile} if there
exists a permutation $\tau\colon\Z_p\to\Z_p$ such that  
$h_1=h_2\circ\tau$.
For $1\le i\le n$ construct increasing functions
\[
 f_i,g_i\colon\Z_p\to\N
\]
such that $f_i$ and $\id_{D_i}*\id_{B_i}$ have the same profile and also that $g_i$ and
$\id_{D'_i}*\id_{B'_i}$ have the same profile. (We define the ordering on $\Z_p$ by
identifying $\Z_p$ with $\{0,1,\dots,p-1\}$.)
Note that, as these functions have the same profiles, \Cref{clm} implies that for any
$1\le i\le n$  and any $j\ge 0$ we have that
\[
 \sum_{x\in\Z_p} \max(0,f_i(x)-j) \le \sum_{x\in\Z_p} \max(0,g_i(x)-j).
\]

\begin{clm}\label{increasing}
 We have
 \[
  \sum_{x\in\Z_p}\prod_i\id_{D_i}*\id_{B_i}(x) \le \sum_{x\in\Z_p}\prod_i f_i(x),
 \]
 and
 \[
  \sum_{x\in\Z_p}\prod_i\id_{D'_i}*\id_{B'_i}(x) = \sum_{x\in\Z_p}\prod_i g_i(x).
 \]
\end{clm}
\begin{proof}[Proof of \Cref{increasing}]
The first part follows from the rearrangement inequality. For the second part let
$\sigma\colon\Z_p\to\Z_p$ be the permutation given by
$(\sigma(p-1),\dots,\sigma(0))=(0,p-1,1,p-2,2,\dots,\frac{p+1}{2},\frac{p-1}{2})$.
It is enough to show the functional equality $g_i(x)=\id_{D'_i}*\id_{B'_i}(\sigma(x))$.
To see why this holds, note that the functions $g_i\colon\Z_p\to\N$ and
$\id_{D'_i}*\id_{B'_i}\circ\sigma\colon\Z_p\to\N$ have the same profile and are increasing.
Indeed, $g_i$ is increasing by construction and $\id_{D'_i}*\id_{B'_i}\circ\sigma$ is increasing
by \Cref{clm}(b); the two functions have the same profile by construction.
\end{proof}

We need one more ingredient to complete the proof of \Cref{main_tool}.

\begin{lem}\label{engine}
 Let $n\in\N$. For $1\le i\le n$ consider increasing functions
 \[
  f_i,g_i\colon\Z_p\to\N
 \]
 with the property that, for all $j\ge0$,
 \[
  \sum_{x\in\Z_p}\max(0,f_i(x)-j) \le \sum_{x\in\Z_p}\max(0,g_i(x)-j).
 \]
 Then it follows that
 \[
  \sum_{x\in\Z_p}\prod_i f_i(x) \le \sum_{x\in\Z_p}\prod_i g_i(x).
 \]
\end{lem}
\begin{proof}
We first claim that for any $i$ and any $r\in\Z_p$, $\sum_{x=r}^{p-1} f_i(x)\le \sum_{x=r}^{p-1} g_i(x)$.
Indeed, assume otherwise and pick the largest $r$ such that $\sum_{x=r}^{p-1} f_i(x)>\sum_{x=r}^{p-1} g_i(x)$.
Then clearly $f_i(r)>g_i(r)$, as otherwise the inequality would also hold at $r+1$.
By setting $j=g_i(r)$ we have
\[
 \sum_{x\in\Z_p}\max(0,f_i(x)-j)\ge\sum_{x=r}^{p-1} (f_i(x)-j)
 >\sum_{x=r}^{p-1}(g_i(x)-j)=\sum_{x\in\Z_p}\max(0,g_i(x)-j),
\]
a contradiction. Now for any increasing function $F\colon\Z_p\to\N$ we have
\begin{align*}
 \sum_{x\in\Z_p}F(x)f_i(x)
 &=F(0)\sum_{x=0}^{p-1}f_i(x)+\sum_{r=1}^{p-1}(F(r)-F(r-1))\sum_{x=r}^{p-1}f_i(x)\\
 &\le F(0)\sum_{x=0}^{p-1}g_i(x)+\sum_{r=1}^{p-1}(F(r)-F(r-1))\sum_{x=r}^{p-1}g_i(x)\\
 &=\sum_{x\in\Z_p}F(x)g_i(x).
\end{align*}
Hence, as $f_1f_2\dots f_{i-1}g_{i+1}\dots g_n$ is increasing,
\[
 \sum_{x\in\Z_p}f_1(x)\dots f_{i-1}(x)f_i(x)g_{i+1}(x)\dots g_n(x)\le
 \sum_{x\in\Z_p}f_1(x)\dots f_{i-1}(x)g_i(x)g_{i+1}(x)\dots g_n(x).
\]
Hence 
\[
 \sum_{x\in\Z_p}f_1(x)\dots f_{n-1}(x)f_n(x)\le \sum_{x\in\Z_p}f_1(x)\dots f_{n-1}(x)g_n(x)\le
 \dots\le \sum_{x\in\Z_p}g_1(x)\dots g_n(x),
\]
as required.
\end{proof}
This concludes the proof of \Cref{main_tool}.
\end{proof}

\section{Passing to the continuous case}\label{continuous}

From \Cref{main_tool} we can easily pass to the one-dimensional 
torus $\T=\R/\Z$. We begin with a useful lemma. 

\begin{lem}\label{lem_close}
 Let\/ $d,n\in\N$ and\/ $0\le\delta\le 1$, and for $i\in [n]$ let\/ $X_i$, $X'_i$, $Y_i$
 and\/ $Y'_i$ be measurable subsets of\/~$\T^d$ \textup{(}or $\R^d$\textup{)} with positive
 \textup{(}finite\textup{)} measure and\/ $|X_i\sdiff X'_i|\le \delta |X_i|$ and\/
 $|Y_i\sdiff Y'_i| \le \delta |Y_i|$. Then for points $y_i$, $y'_i$ chosen uniformly at
 random from the sets $Y_i$ and $Y'_i$, respectively, we have
 \[
  \Big|\E(|\medcup_i(X'_i+y'_i)|)-\E(|\medcup_i(X_i+y_i)|)\Big|\le 4n\delta \sum_i |X_i|.
 \]
\end{lem}
\begin{proof}
Consider a coupling of $y_i$ and $y'_i$ which maximizes the probability $\Prb(y_i=y'_i)$.
It is easy to see that there exists such a coupling where
\[
 \Prb(y_i\ne y'_i)\le \frac{|Y_i\cap Y'_i|}{\max(|Y_i|,|Y'_i|)}.
\]
Indeed, pick elements $a$, $b$, $c$ and $\theta$ uniformly at random from $Y_i\cap Y'_i$,
$Y_i\setminus Y'_i$, $Y'_i\setminus Y_i$ and $[0,1]$, respectively. Set $y_i=a$ if
$\theta\le |Y_i\cap Y'_i|/|Y_i|$ and set $y_i=b$ otherwise; similarly
set $y'_i=a$ if $\theta\le |Y_i\cap Y'_i|/|Y'_i|$ and set $y_i=c$ otherwise. Then
$y_i$ and $y'_i$ are uniform in $Y_i$ and $Y'_i$ respectively and
\[
 \Prb(y_i\ne y'_i)=\Prb\bigg(\theta \ge \frac{|Y_i\cap Y'_i|}{\max(|Y_i|,|Y_i'|)}\bigg)
 =1-\frac{|Y_i\cap Y'_i|}{\max(|Y_i|,|Y'_i|)} \le \frac{|Y_i\sdiff Y'_i|}{\max(|Y_i|,|Y'_i|)}.
\]
Let $E$ be the event that for every $i\in [n]$ we have $y_i=y'_i$. By the above estimate, we have 
\[
 \Prb(E^c) \le \sum_i \Prb(y_i\ne y_i') \le \sum_i \frac{|Y_i\cap Y'_i|}{\max(|Y_i|,|Y_i'|)}
\]
Conditioned on the event $E$, we have the inequality
\[
 \Big||\medcup_i(X'_i+y'_i)|-|\medcup_i(X_i+y_i)|\Big|\le \sum_i |X_i\sdiff X'_i|.
\]
In general, we  have the trivial inequality 
\[
 \Big||\medcup_i(X'_i+y'_i)|-|\medcup_i(X_i+y_i)|\Big|\le \sum_i (|X_i|+|X'_i|).
\]
Combining the last three estimates, we conclude
\begin{align*}
 \Big|\E(|\medcup_i(X'_i+y'_i)|)-\E(|\medcup_i(X_i&+y_i)|)\Big|=
 \Big|\E(|\medcup_i(X'_i+y'_i)|-|\medcup_i(X_i+y_i)|)\Big|\\
 &\le \Big(\sum_i |X_i|+|X'_i|\Big) \Prb(E) + \Big(\sum_i |X_i\sdiff X'_i|\Big)\Prb(E^c)\\
 &\le \Big(\sum_i |X_i|+|X'_i|\Big) \sum_i\frac{|Y_i\sdiff Y'_i|}{\max(|Y_i|,|Y'_i|)} + \sum_i |X_i\sdiff X'_i|\\
 &\le \Big(\sum_i |X_i|+2|X_i|\Big) \sum_i\delta+\sum_i \delta |X_i|\\
 &= (3n+1)\delta \sum_i |X_i|.
\end{align*}
The last inequality follows from the hypotheses that  $|X_i\sdiff X'_i|\le \delta |X_i|$
and $|Y_i\sdiff Y'_i|\le \delta |Y_i|$.
\end{proof}

\begin{cor}\label{main_corollary}
 Let $n\in\N$ and let\/ $A_1,\dots,A_n$ and\/ $B_1,\dots,B_n$ be sets of positive measure in~$\T$.
 Construct intervals $A'_1,\dots,A'_n$ and\/ $B'_1,\dots,B'_n$, each centred about~$0$,
 with the same measure as the sets $A_1,\dots,A_n$ and\/ $B_1,\dots,B_n$, respectively.
 Then for $b_i$ and\/ $b'_i$ chosen independently and uniformly at random from $B_i$ and\/ $B'_i$,
 respectively, we have 
 \[
  \E(|\medcup_i (A_i+b_i)|)\ge \E(|\medcup_i (A'_i+b'_i)|).
 \]
\end{cor}
\begin{proof}
We identify $\T=\R/\Z$ with $(-\frac{1}{2},\frac{1}{2}]$, and fix a small parameter $\eps$ with $0<\eps<1$.
As all $A_i$, $B_i$, $A'_i$ and $B'_i$ are measurable, we can approximate them with sets
$C_i$, $D_i$, $C'_i$ and~$D'_i$, respectively, so that for $i\in [n]$ we have 
\[
 C_i=\bigcup_{j\in [k]}\big(\tfrac{q_{i,j}}{p},\tfrac{r_{i,j}}{p}\big],\qquad
 D_i=\bigcup_{j\in [k]}\big(\tfrac{s_{i,j}}{p},\tfrac{t_{i,j}}{p}\big],\qquad
 C'_i=\big(\tfrac{-h_i}{p},\tfrac{h_i}{p}\big]\quad\text{ and }\quad
 D'_i=\big(\tfrac{-\ell_i}{p},\tfrac{\ell_i}{p}\big],
\]
where $p$ is a prime, $k\in \N$ with $k\le \eps p$, $h_i,\ell_i\in\N$ and
$-\frac{p}{2}<q_{i,1}<r_{i,1}<q_{i,2}<r_{i,2}<\dots< r_{i,k}<\frac{p}{2}$
and $-\frac{p}{2}<s_{i,1}<t_{i,1}<\dots< s_{i,k}<t_{i,k}<\frac{p}{2}$ are integers with
$2h_i=\sum_j (r_{i,j}-q_{i,j})$ and $2\ell_i=\sum_j(t_{i,j}-s_{i,j})$, and finally
$|C_i\sdiff A_i|\le \eps|A_i|$, $|D_i\sdiff B_i|\le \eps|B_i|$,
$|C'_i\sdiff A'_i|\le \eps|A_i'|$ and $|D'_i\sdiff B'_i|\le \eps|B_i'|$.

\begin{clm}\label{clm_approx}
 For $d_i$ and\/ $d'_i$ chosen uniformly at random from $D_i$ and\/ $D'_i$ respectively, we have 
 \[
  \E(|\medcup_i (C_i+d_i)|)\ge \E(|\medcup_i (C'_i+d'_i)|) - 2n\eps.
 \]
\end{clm}
\begin{proof}
For $i\in [n]$, construct subsets of $\Z_p$
\[
 X_i=\bigcup_{j\in [k]}\{q_{i,j}+1,\dots,r_{i,j}\}\quad\text{and}\quad
 Y_i=\bigcup_{j\in [k]}\{s_{i,j}+1,\dots,t_{i,j}\}
\]
and
\[
 X'_i=\{1-h_i,\dots,h_i\}\quad\text{and}\quad Y'_i=\{1-\ell_i,\dots,\ell_i\}.
\]
Note that $|X_i|=|X'_i|$ and $|Y_i|=|Y'_i|$. Moreover, $X'_i$ and $Y'_i$
are discrete intervals centred around $1/2$. Finally, write
\[
 y_i=\lceil pd_i\rceil\quad\text{and}\quad y'_i=\lceil p d'_i\rceil.
\]

Note that if $d_i$, $d'_i$ are chosen uniformly at random from $D_i$, $D'_i$,
respectively, then $y_i$, $y'_i$ are chosen uniformly at random from $Y_i$, $Y'_i$,
respectively. Moreover, for each such choice 
\[
 \big||\medcup_i (C_i+d_i)| - \tfrac{1}{p}|\medcup_i(X_i+y_i)|\big| \le \tfrac{kn}{p}
 \quad\text{and}\quad
 \big||\medcup_i (C'_i+d'_i)| - \tfrac{1}{p}|\medcup_i(X'_i+y'_i)|\big| \le \tfrac{kn}{p}.
\]

By \Cref{main_tool}, we have that
\[
 \E(|\medcup_i (X_i+y_i)|) \ge \E(|\medcup_i (X'_i+y'_i) |).
\]
Combining the last two inequalities, and recalling $k \le \eps p$, we conclude that
\[
 \E(|\medcup_i (C_i+d_i)|) \ge \E(|\medcup_i (C'_i+d'_i)|) - \tfrac{2kn}{p}
 \ge \E(|\medcup_i (C'_i+d'_i)|) - 2n\eps.
 \qedhere
\]
\end{proof}

By \Cref{lem_close} and the fact that $|C_i\sdiff A_i|\le\eps|A_i|$,
$|D_i\sdiff B_i|\le\eps |B_i|$, $|C'_i\sdiff A'_i|\le\eps|A'_i|$
and $|D'_i\sdiff B'_i|\le\eps|B'_i|$, for $b_i$, $b'_i$, $d_i$ and $d'_i$
chosen uniformly at random from $B_i$, $B'_i$, $D_i$ and $D'_i$, respectively, we have
\[
 \Big|\E(|\medcup_i (C_i+d_i)|) - \E(|\medcup_i (A_i+b_i)|)\Big|
 \le 4n\eps \sum_i |A_i| \le 4n^2\eps
\]
and
\[
 \Big|\E(|\medcup_i (C'_i+d'_i)|) - \E(|\medcup_i (A'_i+b'_i)|)\Big|
 \le 4n\eps \sum_i|A'_i| \le 4n^2\eps.
\]

Combining the last inequalities with \Cref{clm_approx}, we have
\[
 \E(|\medcup_i (A_i+b_i)|) \ge \E(|\medcup_i (A'_i+b'_i)|) - 2n\eps - 8n^2\eps.
\]
As $\eps>0$ can be chosen arbitrarily small, we conclude that
\[
 \E(|\medcup_i (A_i+b_i)|) \ge \E(|\medcup_i (A'_i+b'_i)|).
 \qedhere
\]
\end{proof}

\section{Derivation of \Cref{main_torus}}\label{torus}

For $u\in\T^{d-1}$ and any subset $X$ of $\T^d$, write
\[
 X_u=\{x:(x,u)\in X\},
\]
for the section of $X$ over~$u$, where we identify $\T^d$ with $\T\times\T^{d-1}$. 
We define the \emph{Steiner symmetrization} of a measurable set $X$ in $\T^d$ as
the measurable subset $S_d(X)$ of $\T^d$ which has the property that for every
$u\in\T^{d-1}$, $S_d(X)_u$ is a closed interval in $\T$ centred around the origin
with $|S_d(X)_u|=|X_u|$. 
For the sake of definiteness, when the closed interval could be either a single point or the empty set we choose it to
be the empty set. This is purely for convenience, and any other choice would make no difference to any of the arguments.
An alternative would have been to be more careful with this choice, as a way of ensuring that closed sets always do map
to closed sets, for example, but there is no need for this.

More generally, given an integer $d\times d$ matrix $M$ with $\det(M)=1$, i.e., $M\in SL_d(\Z)$,
the Steiner symmetrization $S_M$ transforms a measurable set $X$ in $\T^d$ into the set $X^M=S_M(X)$
in $\T^d$ where $X^M=S_M(X)=M^{-1}(S_d(M(X)))$.
Note that $M$ induces a measure preserving bijection $\T^d\to\T^d$.

\begin{lem}\label{second'}
 For every $M\in SL_d(\Z)$ and for any measurable subsets $X$ and\/ $Y$ of\/ $\T^d$ we have
 \[
  |X^M \sdiff Y^M| \le |X\sdiff Y|.
 \]
\end{lem}
\begin{proof}
As $\det(M)=1$, for any measurable sets $X'$ and $Y'$ we have
\[
 |M(X') \sdiff M(Y')|=|M(X'\sdiff Y')|=|X'\sdiff Y'|.
\] 
Therefore, it is enough to show the result when $M$ is the identity:
\[
 |S_d(X) \sdiff S_d(Y)| \le |X\sdiff Y|.
\]
As $S_d$ acts in each section $\T \times \{u\}$ independently, it is enough
to show the result in dimension $d=1$. So assume $d=1$ and note that
\[
 |S_1(X) \sdiff S_1(Y)| = \big||S_1(X)| - |S_1(Y)|\big| = \big||X|-|Y|\big| \le |X\sdiff Y|.
 \qedhere
\]
\end{proof}

We now start to work towards our result, but just in two dimensions. As we will see,
this is where the work actually comes, and afterwards it will be a straightforward
matter to deduce the result in higher dimensions.

\begin{lem}\label{compress_limit_2}
 There exist matrices $M_j\in SL_2(\Z)$ for $j \in \N$ such that the following holds.
 Let\/ $X$ be a measurable set in $\T^2$ and let\/ $X'=[-c,c]\times\T$
 be a band in $\T^2$ with the same area as~$X$. Define recursively the sets 
 $X_0=X$ and\/ $X_{j+1}=X_j^{M_j}$, each $j \in \N$. Then
 \[
  \lim_{j\to\infty} |X_j \sdiff X'| = 0
 \]
 uniformly in\/~$X$.
\end{lem}
Here, as usual, by `uniformly in $X$' we mean that for each $\eps>0$ there exists
$n$ such that for any $X$ the value of $|X_j \sdiff X'|$ is at most $\eps$ for all $j \ge n$.
\begin{proof}
For $j \in \N$ define $M_j\in SL_2(\Z)$ so that $M_{2j}(x,y)=(x,y)$ and $M_{2j+1}(x,y)=(x,y+jx)$
and note that $M_j(X')=X'$ and hence $(X')^{M_j}=X'$ for all~$j$.
Hence, by \Cref{second'}, $|X_j \sdiff X'|$ decreases as $j$ increases.

Fix $X$ with $|X|=|X'|=2c$ and assume for some $j>0$, $|X_{2j+1}\sdiff X'|>4\eps$.
Note that we must have $2c\in[2\eps,1-2\eps]$ otherwise we automatically have
$|X_{2j+1} \sdiff X'|\le4\eps$. For each $y\in \T$ we have $X'_y=[-c,c]$ and
$X_{2j+1,y}=S_2(X_{2j})_y=[-c(y),c(y)]$ is also an interval in $\T$ centred
at the origin whose width in general will depend on~$y$. Define the sets
\[
 S=\{ y \in \T : c(y) \le c-\eps/2\}\quad\text{and}\quad
 B=\{ y \in \T : c(y) \ge c+\eps/2\}.
\]
Note that $|X_{2j+1}\sdiff X'|>4\eps$ and $|X_{2j+1}|=|X'|$ implies
$|X'\setminus X_{2j+1}|$, $|X_{2j+1}\setminus X'|>2\eps$, and hence
\[
 |S|\ge \eps \quad \text{and}\quad |B| \ge \eps.
\]

By the definition of the Steiner symmetrization $S_{M_{2j+1}}$, we note that for every
$t \in \T$ the set $I'_t=\{x:(x,t-jx)\in X_{2j+2}\}$ is an interval in $\T$ centred
at the origin with the same size as the set $I_t=\{x:(x,t-jx)\in X_{2j+1}\}$. Moreover,
for every $t \in \T$ we have that the set $\{x:(x,t-jx)\in X'\}$ is equal to $[-c,c]$. which we will call
$I$ for convenience. 

\begin{clm}\label{increment}
 For every $j\ge 2/\eps$ and $t\in \T$ we have 
 \[
  |I'_t\cap I| \ge |I_t\cap I|+ j^{-1}\eps,
 \]
\end{clm}
\begin{proof}
As $I$ and $I'_t$ are centred intervals with $|I'_t|=|I_t|$, it is enough to show that
\[
 | I \setminus I_t | \ge j^{-1}\eps
 \quad\text{and}\quad
 | I_t \setminus I | \ge j^{-1}\eps.
\]

Consider the functions $f \colon [c-j^{-1},c) \to \T$ and $g\colon (c,c+j^{-1}] \to \T$
defined by $f(x)=g(x)=t-jx$ and note that these functions are bijective and scale
the measure by a factor~$j$. Hence if we set $S'=f^{-1}(S)$ and $B'=g^{-1}(B)$ then
\[
 |S'|\ge j^{-1}\eps \quad\text{and}\quad |B'|\ge j^{-1}\eps.
\]

As $j^{-1}\le \eps/2$ and $2c\in[2\eps,1-2\eps]$, we have
$c-j^{-1}>-c$ and $c+j^{-1}<1-c$, so $S' \subseteq I=[-c,c]$ and $B' \cap I=\emptyset$
in~$\T$. Therefore, it is enough to show that
\[
 S' \cap I_t = \emptyset
 \quad\text{and}\quad
 B' \subseteq I_t.
\]

Assume for a contradiction that
\[
 x \in S' \cap I_t,
\]
and note that if we let $y=t-jx$ then this assumption is equivalent to
\[
 x \in S' \quad\text{and}\quad (x,y) \in X_{2j+1}.
\]
As $x \in S'$ and $y=f(x)$, we have $y \in S$. By construction of~$S$,
we have $x \in [-c+\eps/2, c-\eps/2]$. However, by construction
$x \in S' \subseteq [c-j^{-1},c)$. This gives the desired contradiction
as $j^{-1}\le \eps/2$. We conclude that
\[
 S' \cap I_t = \emptyset.
\]
Analogously, we conclude that
\[
 B' \subseteq I_t.
\]
This finishes the proof of the claim.
\end{proof}

To conclude our proof, note that \Cref{increment} implies that for every $j\ge 2/\eps$
either $|X_{2j+1}\sdiff X'|\le 4\eps$, or 
\[
 |X_{2j+2} \cap X'| \ge |X_{2j+1} \cap X'| + j^{-1}\eps,
\]
which is equivalent to
\[
 |X_{2j+2} \sdiff X'| \le |X_{2j+1} \sdiff X'| - 2j^{-1}\eps.
\]
Again by \Cref{second'}, $|X_j \sdiff X'|$ decreases as $j$ increases, which implies that
\[
 |X_{2j+2} \sdiff X'| \le |X_{2j} \sdiff X'| - 2j^{-1}\eps.
\]
As $\sum_j j^{-1}$ diverges, there must exist some $j_0$ depending only on $\eps$
such that $|X_j\sdiff X'|\le 4\eps$ for all $j\ge j_0$, and we note that $j_0$
can be chosen independently of~$X$. We conclude that 
\[
 \lim_{j \to \infty} |X_j \sdiff X'| =0\quad\text{uniformly in }X.
 \qedhere
\]
\end{proof}

\begin{cor}\label{compress_limit_d}
 There exists matrices $M_j\in SL_d(\Z)$ for $j \in \N$ such that the following holds.
 Let\/ $X$ be a measurable set in $\T^d$ and let\/ $X'=[-c,c]\times\T^{d-1}$
 be a band in $\T^d$ with the same volume as~$X$. Define recursively 
 the sets $X_0=X$ and\/ $X_{j+1}=X_j^{M_j}$, each $j \in \N$. Then
 \[
  \lim_{j\to\infty} |X_j \sdiff X'| = 0
 \]
 uniformly in $X$.
\end{cor}
\begin{proof}
Let $M_{j,k}\in SL_d(\Z)$ be the matrix that applies the $M_j$ from \Cref{compress_limit_2}
to coordinates 1 and $k$ and leaves all other coordinates fixed.
Using the sequence $M_{j,2}$, $j=1,\dots,n_1$, we have that for sufficiently large
$n_1$, $|X_{n_1}\sdiff X^{(2)}|<\eps$, where $X^{(2)}$ is a set with $|X^{(2)}|=|X|$
that is independent of the 2nd coordinate, i.e., $(x_1,x_2,\dots,x_d)\in X^{(2)}$
iff $(x_1,x'_2,\dots,x_d)\in X^{(2)}$ for any $x_2,x'_2\in\T$.
Continuing with the sequence $M_{j,3}$, $j=1,\dots,n_1$, and recalling \Cref{second'},
we have that for sufficiently large $n_1$,
$|X_{2n_1}\sdiff X^{(3)}|\le |X_{n_1}\sdiff X^{(2)}|+|X^{(2)}_{n_1}\sdiff X^{(3)}|<2\eps$,
where $X^{(2)}_{n_1}$ is the result of applying these symmetrizations to $X^{(2)}$ and
$X^{(3)}$ is independent of both 2nd and 3rd coordinates. Continuing this process for each coordinate
in turn gives a sequence such that $|X_{(d-1)n_1}\sdiff X^{(d)}|<(d-1)\eps$ for sufficiently
large~$n_1$. But clearly $X^{(d)}=X'$. Concatenating this sequence of matrices with
progressively larger $n_1<n_2<\cdots$ and noting that $|X_j\sdiff X'|$ can never increase,
we have a sequence of matrices such that $|X_j\sdiff X'|\to 0$ uniformly in~$X$.
\end{proof}

\begin{lem}\label{thm_compression_T'}
 Let $d,n\in\N$, $M\in SL_d(\Z)$ and let\/ $X_1,\dots,X_n$ and\/ $Y_1,\dots,Y_n$ be
 sets of positive measure in~$\T^d$. For $i\in[n]$, let sets $X_i^M$ and\/ $Y_i^M$ be
 obtained from the sets $X_i$ and\/ $Y_i$, respectively, by applying the Steiner
 symmetrization~$S_M$. Then for points $y_i$ and\/ $y^M_i$ chosen independently and
 uniformly at random from $Y_i$ and\/ $Y^M_i$, respectively, we have 
 \[
  \E(|\medcup_i(X_i+y_i)|) \ge \E(|\medcup_i(X^M_i+y^M_i)|).
 \]
\end{lem}
\begin{proof}
We may assume without loss of generality that $M$ is the identity.
Recall that for $u\in\T^{d-1}$ and any subset $S$ of $\T^d$, $S_u=\{x:(x,u)\in S\}$,
where we identify $\T^d$ with $\T\times\T^{d-1}$. For $i\in [n]$, write $y_i=(z_i,u_i)$ and
$y^M_i=(z^M_i,u^M_i)$ where $z_i,z^M_i\in\T$ and $u_i,u^M_i\in\T^{d-1}$.
Now $u_i$ and $u^M_i$ have the same distribution as $|Y_{i,u}|=|Y^M_{i,u}|$ for all $u\in\T^{d-1}$.
Thus we can couple $y_i$ and $y^M_i$ so that $u^M_i=u_i$ and, conditioned on~$u_i$,
$z_i$ and $z^M_i$ are independent and uniform in $Y_{i,u_i}$ and $Y^M_{i,u_i}$ respectively.
Also choose $(z,u)\in\T\times\T^{d-1}$ uniformly at random and independent from everything else. Then
\begin{align*}
 \E(|\medcup_i(X_i+y_i)|)&=\Prb((z,u)\in\medcup_i(X_i+(z_i,u_i)))\\
 &=\Prb(z\in\medcup_i(X_{i,u-u_i}+z_i))=\E(|\medcup_i(X_{i,u-u_i}+z_i)|)
\end{align*}
and
\begin{align*}
 \E(|\medcup_i(X^M_i+y^M_i)|)&=\Prb((z,u)\in\medcup_i(X^M_i+(z^M_i,u_i)))\\
 &=\Prb(z\in\medcup_i(X^M_{i,u-u_i}+z^M_i))=\E(|\medcup_i(X^M_{i,u-u_i}+z^M_i)|).
\end{align*}
We now condition on $u,u_1\dots,u_n\in\T$ and show that if we choose $z_i$ and $z_i^M$ uniformly
at random from $Y_{i,u_i}$ and $Y^M_{i,u_i}$ respectively for all $i\in [n]$, then
\[
 \E(|\medcup_i(X_{i,u-u_i}+z_i)|) \ge \E(|\medcup_i(X^M_{i,u-u_i}+z^M_i)|).
\]
Write $A_i=X_{i,u-u_i}$ and $B_i=Y_{i,u_i}$ for $i\in[n]$. By the definition of symmetrization
$A'_i=X^M_{i,u-u_i}$ and $B'_i=Y^M_{i,u_i}$ are intervals centred around $0$ with the same size
as $A_i$ and $B_i$ respectively. Hence, by \Cref{main_corollary}, we can conclude
\[
 \E(|\medcup_i (A_i+b_i)|)\ge \E(|\medcup_i (A'_i+b'_i)|),
\]
Where $b_i=z_i$ and $b'_i=z^M_i$ are chosen uniformly
at random from $B_i=Y_{i,u_i}$ and $B'_i=Y^M_{i,u_i}$ respectively.
\end{proof}

\begin{thm}\label{lem_d}
 Let\/ $d,n\in \N$ and let\/ $A_1,\dots,A_n$ and\/ $B_1,\dots,B_n$ be sets of positive
 measure in~$\T^d$. For $i\in[n]$, let $A'_i$ and\/ $B'_i$ be bands in $\T^d$ centred at the
 origin with the same volumes as $A_i$ and\/ $B_i$ respectively, i.e., sets of the form
 $I \times \T^{d-1}$, where $I$ is a closed interval in $\T$ centred at the origin.
 Then for $b_1,\dots,b_n$, $b'_1,\dots,b'_n$ chosen independently and uniformly at random from
 $B_1,\dots,B_n$ and\/ $B'_1,\dots,B'_n$, respectively, we have 
 \[
  \E(|\medcup_i(A_i+b_i)|) \ge \E(|\medcup_i(A'_i+b'_i)|).
 \]
\end{thm}
\begin{proof}
For $j \in \N$ let $M_j\in SL_d(\Z)$ be the matrices given by \Cref{compress_limit_d}.
For $i \in [n]$ set $A_{i,0}=A_i$ and $B_{i,0}=B_i$ and for $j \in \N$ recursively define
the sequences $A_{i,j+1}=A_{i,j}^{M_j}$ and $B_{i,j+1}=B_{i,j}^{M_j}$. For $i \in [n]$ and
$j \in \N$ choose points $b_{i,j}$ independently and uniformly at random from $B_{i,j}$.
Then by \Cref{thm_compression_T'} we have
\[
 \E(|\medcup_i(A_{i,j}+b_{i,j})|) \ge \E(|\medcup_i(A_{i,j+1}+b_{i,j+1})|)
\]
for all $j\in\N$. But by \Cref{compress_limit_d} we have
\[
 \lim_{j\to\infty} |A_{i,j} \sdiff A'_i|=\lim_{j \to \infty} |B_{i,j} \sdiff B'_i|=0
\]
which, by \Cref{lem_close}, implies
\[
 \lim_{j\to\infty}\E(|\medcup_i(A_{i,j}+b_{i,j})|) = \E(|\medcup_i(A'_i+b'_i)|).
\]
Thus we conclude that
\[
 \E(|\medcup_i(A_i+b_i)|)=\E(|\medcup_i(A_{i,0}+b_{i,0})|) \ge \E(|\medcup_i(A'_i+b'_i)|).
 \qedhere
\]
\end{proof}

Note that \Cref{main_torus} now follows immediately from \Cref{lem_d}, by taking all $A_i=A$ and all $B_i=B$.

We mention in passing another approach to proving \Cref{main_torus}. In the above, we have started with the
result in the discrete setting of $\Z_p$, namely \Cref{main_tool}. We then passed to the circle
(\Cref{main_corollary}) and went from there to the result in the torus. Alternatively, one could
use a result of Tao (Theorem~1.1 in~\cite{Tao}), which is the analogue in compact connected abelian
groups of Theorem~2 in \cite{pollard} (the result we used in Section~\ref{tool}), to get to \Cref{main_torus} by
using it instead of the Pollard-type inequalities from \cite{pollard} in (a generalisation of) \Cref{main_tool}.

\section{Derivation of \Cref{main_sphere}}

For any subset $X$ of $\R^d$ and $u\in \R^{d-1}$, we define the section $X_u=\{x:(x,u)\in X\}$, where we identify
$\R^d$ with $\R\times\R^{d-1}$. We define the \emph{Steiner symmetrization} $S_d(X)$ of a measurable subset $X$ of
$\R^d$ by setting $S_d(X)_u$ to be a closed interval in $\R$ centred around the origin such that
$|S_d(X)_u|=|X_u|$.

More generally, given a matrix $M\in SL_d(\R)$, the Steiner symmetrization $S_M$ transforms a measurable
set $X$ in $\R^d$ into the set $X^M=S_M(X)$ in $\R^d$ where $S_M(X)=M^{-1}(S_d(M(X)))$.

\begin{lem}\label{thm_compression_R}
 Let $d,n\in \N$, $M\in SL_d(\R)$ and let $X$ and $Y$ be sets of positive finite measure in~$\R^d$.
 Let sets $X^M$ and $Y^M$ be obtained from sets $X$ and~$Y$, respectively, by applying the Steiner
 symmetrization~$S_M$. Then for points $y_1,\dots,y_n$, $y^M_1,\dots,y^M_n$ chosen independently
 and uniformly at random from $Y$, $Y^M$, respectively, we have 
 \[
  \E(|\medcup_i(X+y_i)|) \ge \E(|\medcup_i(X^M+y^M_i)|).
 \]
\end{lem}
\begin{proof}
We may assume without loss of generality that $M$ is the identity. In addition, by scaling $X$ and $Y$
by the same sufficiently small parameter $\alpha>0$, we may assume that $X$ and $Y$, as well as $X^M$
and~$Y^M$, are measurable subsets of $[-\frac{1}{4},\frac{1}{4})^d \subseteq \R^d$. (By \cref{lem_close} we may approximate $X$ and $Y$ by
bounded sets.)
In particular, $X+Y, X^M+Y^M \subseteq [-\frac{1}{2},\frac{1}{2})^d$.
As the canonical map $\pi\colon [-\frac{1}{2},\frac{1}{2})^d \to \T^d$ is a linear isometry,
\Cref{thm_compression_R} for sets $X$ and $Y$ reduces to \Cref{thm_compression_T'} for sets $\pi(X)$
and $\pi(Y)$.
\end{proof}

\begin{proof}[Proof of \Cref{main_sphere}]
By \cite{volcic}, there exist a sequence of Steiner symmetrizations which we can apply inductively
to the measurable sets $A=A^0$ and $B=B^0$ in order to produce measurable sets $A^k$ and $B^k$ for
$k \in \N$, with the property that $|A^k \sdiff A'|, |B^k \sdiff B'| \to 0$ as $k \to \infty$.
Choose points $b^k_1,\dots, b^k_n$ uniformly at random from $B^k$.

By \Cref{thm_compression_R}, for $k\in \N$ we have
\[
 \E(|\medcup_i(A^k+b^k_i)|) \ge \E(|\medcup_i(A^{k+1}+b^{k+1}_i)|).
\]
We also have by \Cref{lem_close}
\[
 \lim_{k\to\infty}\E(|\medcup_i(A^k+b^k_i)|) = \E(|\medcup_i(A'+b'_i)|).
\]
Combining the last two equations, we conclude that
\[
 \E(|\medcup_i(A+b_i)|) = \E(|\medcup_i(A^0+b^0_i)|) \ge \E(|\medcup_i(A'+b'_i)|).
 \qedhere
\]
\end{proof}

Actually, the proof above can be easily adapted to give the following generalization.
\begin{thm}
 Let $d,n \in \N$ and let $A_1,\dots,A_n$, $B_1,\dots,B_n$ be measurable sets in~$\R^d$
 of finite positive measure. For $i\in[n]$ let $A'_i$ and\/ $B'_i$ be balls in $\R^d$
 centred at the origin with the same volumes as $A_i$ and\/ $B_i$, respectively.
 Then for $b_i$, $b'_i$ chosen independently and uniformly at random from $B_i$, $B'_i$
 respectively, we have 
 \[
  \E(|\medcup_i(A_i+b_i)|) \ge \E(|\medcup_i(A'_i+b'_i)|).
 \]
\end{thm}

\section{Note on the size of $A$ + $n$ random points of $B$}\label{calculations}

In this short section we give the estimates needed to
prove \Cref{estimate_intervals} and \Cref{estimate_sphere}.

\begin{proof}[Proof of \Cref{estimate_intervals}]
By \Cref{main_torus}, we may assume that $A$ and $B$ are intervals with $A=[0,a]$ and $B=[0,b]$, say.
Order the points $\{b_1,\dots,b_n\}$ so that $0<b_1<b_2<\dots<b_n<b$. Then it is easy to see that,
for $i>1$, $b_i-b_{i-1}$ has the same distribution as $b_1$. Indeed, if we condition on $b_i$ the set
$\{b_1,\dots,b_{i-1}\}$ is just a uniformly chosen random set of $i-1$ points in $[0,b_i]$ and so by
symmetry $b_i-b_{i-1}$ has the same distribution as $b_1-0=b_1$. As this holds conditioned on $b_i$
for any choice of $b_i$, it also holds unconditionally. We also note that the distribution of $b_1$
is just the minimum of $n$ i.i.d.{} random variables that are uniform on the interval $[0,b]$.

Clearly $A+b_1=[b_1,a+b_1]$ has length~$a$. Now each $A+b_i$, $i>1$ adds an additional length
of $\min\{a,b_i-b_{i-1}\}$. Indeed, it adds the interval $(a+b_{i-1},a+b_i]$ except when
$b_i-b_{i-1}>a$, in which case it adds the interval $[b_i,a+b_i]$. Hence, provided $a+b\le 1$
so that there is no wrap around the torus, linearity of expectation and the above result on the
distribution of $b_i-b_{i-1}$ implies
\[
 \E(|\medcup_i (A+b_i)|)
 =a+(n-1)\E(\min\{a,b_1\})
 =a+(n-1)b\cdot \Prb(b_*\le\min\{a,b_1\}),
\]
where $b_*$ is chosen independently and uniformly at random from $[0,b]$.
But this last probability is just
the probability that on picking $n+1$ points $\{b_*,b_1,\dots,b_n\}$ at random in $[0,b]$  
that $b_*$ is the smallest, and that the smallest is at most~$a$. By symmetry this is $1/(n+1)$ times the
probability that the smallest of $n+1$ uniform numbers in $[0,b]$ is at most $a$, which is 1 if
$a\ge b$ and $1-(1-a/b)^{n+1}$ if $a<b$. The result follows for $a+b\le1$.

For $a+b>1$ we need to subtract the overlap $X:=|A+(b_n-1)\cap A+b_1|$ that can occur as a result of wrap around.
We note that the result holds for $n=1$, so assume $n\ge 2$.
By reflecting the interval $[b_1,b]$ we see that $b-b_n$ has the same distribution as $b_2-b_1$,
even conditioned on~$b_1$. Thus $X$ has the same distribution as $\max\{a+b-1-b_2,0\}$. 
It is easy to see that the pdf of $b_2$ is $n(n-1)x(b-x)^{n-2}b^{-n}$ and thus $\E X$ can be
calculated (after some algebra) as
\[
 \int_{0}^{a+b-1}\mkern-20mu n(n-1)x(b-x)^{n-2}b^{-n}(a+b-1-x)\,dx
 =a+\tfrac{n-1}{n+1}b-1+\big(\tfrac{1-a}{b}\big)^n\big(b-\tfrac{n-1}{n+1}(1-a)\big).
\]
The result then follows after some algebraic simplifications.
\end{proof}

\begin{proof}[Proof of \Cref{estimate_sphere}]
By \Cref{main_sphere}, we may assume $A$ and $B$ are balls centred at the origin.
If $b_i\in B$ then a point $x\in A+B$ is covered by $A+b_i$ if and only if $b_i\in (x-A)\cap B$.
Hence for a uniformly random choice of $b_i$,
\[
 \Prb(x\in A+b_i)=\tfrac{|(x-A)\cap B|}{|B|}.
\]
The expected volume $A+B$ \emph{not} covered by $\medcup_i(A+b_i)$ is then given simply as
\[
 \int_{A+B} \prod_{i=1}^n(1-\Prb(x\in A+b_i))\,dx=\int_{A+B}\big(1-\tfrac{|(x-A)\cap B|}{|B|}\big)^n\,dx.
\]
As we are considering $n$ very large, the integral will dominated by the region very close to the boundary
of the ball $A+B$ where $(x-A)\cap B$ is very small. Let the radii of $A$ and $B$ be $a$ and $b$ respectively,
and assume $x$ is at distance $a+b-\eps$ from the centre of $A+B$. Then
\begin{equation}\label{e:px}
 \frac{|(x-A)\cap B|}{|B|}=\frac{a^d\int_0^\alpha \sin^d\theta\,d\theta
 +b^d\int_0^\beta \sin^d\theta\,d\theta}{b^d\int_0^\pi \sin^d\theta\,d\theta},
\end{equation}
where the triangle with side lengths $a$, $b$ and $a+b-\eps$ forms angles $\alpha$ and $\beta$ between
the side of length $a+b-\eps$ and the sides of length $a$ and $b$ respectively. To see this note that
$(x-A)\cap B$ is formed from two caps, a cap of $x-A$ subtending an angle $\alpha$ at its centre and a
cap of $B$ subtending an angle $\beta$ at its centre. The integrals that give the volumes of these caps
is of the form $\int_{a\cos\alpha}^a v_{d-1}(\sqrt{a^2-t^2})^{d-1}\,dt$ for $x-A$, and similarly for $B$ with
$v_{d-1}$ equal to the volume of a unit $(d-1)$-sphere. A change of variables $t=a\cos\theta$
puts this in the form $v_{d-1}a^d\int_0^\alpha \sin^d\theta\,d\theta$, and similarly for the cap of~$B$.

Now fix $a$, $b$ and $d$ and consider the case when $\eps\to 0$. We have $a\sin\alpha=b\sin\beta$
and $a\cos\alpha+b\cos\beta=a+b-\eps$. As $\eps\to0$ this gives $\alpha^2=(1+o(1))\frac{2\eps b/a}{a+b}$,
$\beta^2=(1+o(1))\frac{2\eps a/b}{a+b}$ and, for $0\le \theta\le\max\{\alpha,\beta\}$,
$\sin^d\theta = (1+o(1))\theta^d$. Thus
\[
 \frac{|(x-A)\cap B|}{|B|}
 =\frac{(1+o(1))(a^{-1}+b^{-1})(2\eps ab)^{(d+1)/2}}{(d+1)(a+b)^{(d+1)/2} b^d(v_d/v_{d-1})}
 =(1+o(1))C_{a,b,d}\eps^{(d+1)/2},
\]
where $C_{a,b,d}$ is a positive constant that depends on $a$, $b$ and $d$, and the $o(1)$ is as $\eps\to0$.
The expected fraction of $A+B$ not covered by $\medcup_i(A+b_i)$ is then
\begin{align*}
 \frac{1}{|A+B|}\int_{A+B}\big(1-\tfrac{|(x-A)\cap B|}{|B|}\big)^n\,dx
 &=(1+o(1))\tfrac{d}{a+b}\int_0^\infty \exp\big(-C_{a,b,d}n\eps^{(d+1)/2}\big)\,d\eps\\
 &=(1+o(1))\tfrac{d}{a+b}\Gamma(\tfrac{d+3}{d+1})(C_{a,b,d}n)^{-2/(d+1)}.
\end{align*}
Indeed, both integrals are dominated by the range when $\eps$ is very small, in which case we can approximate
the $(1-\cdots)^n$ term by an exponential and the volume of a spherical shell of thickness $d\eps$ by $\frac{d}{a+b}d\eps$
times the volume of the sphere $A+B$. Evaluating the final integral gives the expected fraction of $A+B$ uncovered
as $(c+o(1))n^{-2/(d+1)}$ for some constant $c>0$ depending on $a$, $b$ and~$d$. Clearly $c$ is unchanged by
uniform scaling of $A$ and $B$ and hence $c$ depends only on the dimension and the ratio of
the volumes of $A$ and~$B$.
\end{proof}

We note that for the case when there is a bounded ratio between the radii $a$ and $b$, one has
$c=(\frac{b}{2a}+o(1))d=\Theta(d)$ as $d\to\infty$. Thus we need
$n$ to be at least of order $d^{(d+1)/2}$ for the asymptotic formula in \Cref{estimate_sphere}
to be a reasonable guide to the size of $\E|A+\{b_1,\dots,b_n\}|$.

\section{Sharp Stability Result  for \Cref{main_torus}}\label{stability}

Our aim in this section is to prove a stability result for \Cref{main_torus}. 
Recall that a \emph{Bohr set} $X$ in $\T^d$ is a set of the form
$X=\{x \in \T^d \colon |\phi(x)-c| \le \rho \}$, where $\phi \colon \T^d \to \T$
is a continuous homomorphism, $c \in \T$ and $\rho \in [0,\frac{1}{2}]$.
Two Bohr sets in $\T^d$ are said to be parallel if they use the same
homomorphism~$\phi$. The result we will prove, which is sharp, is as follows.

\begin{thm}\label{stab}
 For $\eps>0$ and\/ $d,n\in\N$, there exists $K=K(\eps,n)>0$ such that the following holds.
 Let\/ $A$ and\/ $B$ be measurable subsets of\/ $\T^d$ with $\eps\le |A|,|B|\le 1-\eps$.
 Let\/ $A'$ and\/ $B'$ be intervals in $\T$ centred around\/ $0$ with $|A'|=|A|$ and\/ $|B'|=|B|$.
 Assume that if we choose $b_1,\dots,b_n$ and\/
 $b_1',\dots,b_n'$ independently and uniformly at random from $B$ and~$B'$, respectively, then 
 \[
  \E(|A+\{b_1,\dots,b_n\}|)\le \delta+\E(|A'+\{b'_1,\dots,b'_n\}|).
 \]
 Then, there exist parallel Bohr sets $X$ and\/ $Y$ in $\T^d$ such that
 \[
  |A\sdiff X|,|B\sdiff Y|\le K\delta^{1/2}.
 \]
\end{thm}

We remark that the exponent $1/2$ in \Cref{stab} is optimal as the following example shows.
Fix $\alpha>0$ small and let $A=[0,\frac{1}{2}]$ and
$B=[0,\frac{1}{2}-\alpha] \cup [\frac{1}{2},\frac{1}{2}+\alpha]$,
considered as a subset of~$\T^d$ with $d=1$.
Then $A'=B'=[-\frac{1}{4},\frac{1}{4}]$. It is a simple computational
task to check that
\[
 \E(|A+\{b_1,b_2\}|)= \Theta(\alpha^2)+\E(|A'+\{b_1',b_2'\}|)
\]
and for any Bohr set $X$ in $\T$
\[
 |B\sdiff X|\ge \alpha.
\]
A corresponding statement also holds with
$A=[0,\frac{1}{2}-\alpha] \cup [\frac{1}{2},\frac{1}{2}+\alpha]$ and $B=[0,\frac{1}{2}]$.

We will use some inequalities of Riesz--Sobolev type, and stability versions of these
inequalities, which are due to Christ and Iliopoulou~\cite{Christ}. We state them below just
in the form we need them, namely for $\T^d$, but we mention that, as proved in~\cite{Christ},
they also have more general formulations, being true in any compact connected abelian group.
Here are these two results.

\begin{thm}[\cite{Christ}]\label{Riesz-Sobolev}
 Let\/ $A$, $B$ and\/ $C$ be compact subsets of $\T^d$. Let $A'$, $B'$ and $C'$ be intervals
 in $\T$ centred around\/ $0$ with the same measures as $A$, $B$ and\/~$C$, respectively. Then
 \[
  \int_{C}\id_{A}*\id_{B}(x)\,dx \le \int_{C'}\id_{A'}*\id_{B'}(x)\,dx.
 \]
\end{thm}

\begin{thm}[\cite{Christ}, Theorem 1.3]\label{Stab_Riesz-Sobolev}
 For $\eps>0$ there exists $c=c(\eps)>0$ such that the following holds.
 Let\/ $A$, $B$ and\/ $C$ be compact subsets of\/ $\T^d$ with $\eps\le |A|,|B|,|C|\le 1-\eps$,
 $|A|+|B|+|C|\le 2-\eps$ and\/ $2\max(|A|,|B|,|C|) \le |A|+|B|+|C|-\eps$.
 Let\/ $A'$, $B'$ and\/ $C'$ be intervals in $\T$ centred around\/ $0$ with the same
 measures as $A$, $B$ and\/ $C$, respectively. Assume that
 \[
  \int_{C}\id_{A}*\id_{B}(x)\,dx \ge -\delta +\int_{C'}\id_{A'}*\id_{B'}(x)\,dx.
 \]
 Then there exist parallel Bohr sets $X$, $Y$ and\/ $Z$ in $\T^d$ such that
 \[
  |A\sdiff X|, |B\sdiff Y|, |C\sdiff Z| \le c\delta^{1/2}.
 \]
\end{thm}

The main new ingredient we need is the following result.

\begin{thm}\label{Main_ingredient}
 For $\eps>0$ and $d,n \in \N$ there exists $\tau=\tau(\eps,n)$ such that the following holds.
 Let $A$ and\/ $B$ be compact sets in $\T^d$ with $2\eps \le |A|,|B| \le 1-2\eps$, and let $A'$ and $B'$ be intervals in $\T$ centred at the origin with the same measures
 as $A$ and\/ $B$ respectively.
 
 Assume that
 \[
  \int_{C}\id_{A}*\id_{B}(x)\,dx \le -\alpha +\int_{C'}\id_{A'}*\id_{B'}(x)\,dx
 \]
 for every compact set\/ $C\subseteq\T^d$ with  $\eps \le |C| \le 1-\eps$,
 $|A|+|B|+|C| \le 2-\eps$ and\/ $2\max(|A|,|B|,|C|) \le |A|+|B|+|C|-\eps$, where $C'$is the interval in $\T$ centred at the origin with the same measure as $C$. Then 
 \[\int_{\T^d}(\id_{A}*\id_{B}(x))^n dx
  \le 
  -\tau\alpha +\int_{\T}(\id_{A'}*\id_{B'}(x))^n dx.
 \]
\end{thm}

We shall derive this result from the following general result about arbitrary functions. 

\begin{thm}\label{cont_arbitrary_functions}
 Let $f,g \colon [0,1] \to [0,1]$ be continuous decreasing functions with
 \[
  \int_0^t f(x)\,dx \le \int_0^t g(x)\,dx
 \]
 for $0 \le t \le 1$. Assume that there exists an interval $I \subseteq [0,1]$
 with $g'(t)=-\gamma$ and
 \[
  \int_0^t f(x)\,dx \le -\alpha+\int_0^t g(x)\,dx
 \]
 for all\/ $t\in I$. Then
 \[
  \int_0^{1} f(x)^n dx \le -(\gamma |I|)^{n-1}\alpha+\int_0^{1}g(x)^n dx.
 \]
\end{thm}
\begin{proof}
Define $F(t)=\int_0^t f(x)\,dx$ and $G(t)=\int_0^t g(x)\,dx$ and assume $h\colon [0,1]\to[0,\infty)$
is decreasing and continuously differentiable. Then using integration by parts
\begin{align*}
 \int_0^{1} (g(x)-f(x))h(x)dx
 &=(G(x)-F(x))h(x)\Big|_0^{1}-\int_0^{1} (G(x)-F(x))h'(x)dx\\
 &\ge -\alpha \int_I h'(x)=\alpha(h(a)-h(b)),
\end{align*}
where $I=(a,b)$ and we have used that 
$G(0)=F(0)$, $G(x)-F(x)\ge 0$ and $G(x)-F(x)\ge\alpha$ on $I$. 

Apply this to the function $h(x)=g(x)^{n-1}+g(x)^{n-2}f(x)+\dots+f(x)^{n-1}$
and note that $h$ is decreasing and $g(a)-g(b)\ge \gamma(b-a)$, so
\[
 h(a)-h(b)\ge g(a)^{n-1}-g(b)^{n-1}\ge g(a)^{n-2}(g(a)-g(b))\ge (\gamma|I|)^{n-1}.
\]
We deduce that
\[
 \int_0^{1}g(x)^ndx-\int_0^{1}f(x)^ndx=\int_0^{1}(g(x)-f(x))h(x)\,dx
 \ge \alpha(h(a)-h(b))\ge (\gamma|I|)^{n-1}\alpha.
 \qedhere
\]
\end{proof}

\begin{proof}[Proof of \Cref{Main_ingredient}]
Given a continuous function $h \colon \T^d \to [0,1]$,
the \emph{decreasing rearrangement} of $h$ is the unique continuous decreasing
function $\tilde h\colon [0,1] \to [0,1]$ such that for every $\lambda \in [0,1]$ 
\[
 |\{x : h(x) \ge \lambda\}|=|\{x : \tilde h(x) \ge \lambda\}|.
\]

Note that $\id_{A}*\id_{B} \colon \T^d \to [0,1]$ and $\id_{A'}*\id_{B'} \colon \T \to [0,1]$
are continuous functions. Let $f$ and $g$ be the decreasing rearrangements of $\id_{A}*\id_{B}$
and $\id_{A'}*\id_{B'}$, respectively. For $x\in [-\frac{1}{2},\frac{1}{2})=:\T$, We have
\[
 \!\!\id_{A'}*\id_{B'}(x)=\begin{cases}
  \min(|A|,|B|), & \text{if } 2|x| \le ||A|-|B||;\\
  \max(0,|A|+|B|-1), & \text{if } 2|x| \ge \min(|A|+|B|,2-|A|-|B|);\\
  \frac{|A|+|B|}{2}-x, & \text{otherwise.}
 \end{cases}
\]
For $x \in [0,1]$, we have
\[
 \qquad g(x)=\begin{cases}
  \min(|A|,|B|), & \text{if } x \le ||A|-|B||;\\
  \max(0,|A|+|B|-1), & \text{if } x \ge \min(|A|+|B|,2-|A|-|B|);\\
  \frac{|A|+|B|}{2}-\frac{x}{2}, & \text{otherwise.}
 \end{cases}
\]

For $t \in [0,1]$ set $C'_t=[-\frac{t}{2},\frac{t}{2}]\subseteq\T$.
As $\id_{A'}*\id_{B'}(x) \ge \id_{A'}*\id_{B'}(y)$ for all $x \in C'_t$
and $y \notin C'_t$, it follows that
\[
 \int_{C_t'} \id_{A'}*\id_{B'}(x)\,dx= \int_0^t g(x)\,dx.
\]
As $f$ is a decreasing rearrangement of $\id_{A}*\id_{B}$, it follows that for $t \in [0,1]$
there exists a compact set $C_t$ in $\T^d$ with $|C_t|=t$ such that 
\[
 \int_{C_t} \id_{A}*\id_{B}(x)\,dx= \int_0^t f(x)\,dx.
\]
For $t \in [0,1]$, $C'_t$, $A'$ and $B'$ are centred intervals with the same measures
as the compact sets $C_t$, $A$ and $B$ in $\T^d$, respectively, so \Cref{Riesz-Sobolev} implies that 
\[
 \int_0^t f(x)\,dx=\int_{C_t} \id_{A}*\id_{B}(x)\,dx
 \le \int_{C_t'} \id_{A'}*\id_{B'}(x)\,dx= \int_0^t g(x)\,dx.
\]
Set $I=[||A|-|B||+\eps,\min(|A|+|B|,2-|A|-|B|)-\eps]$.
As $2\eps \le |A|,|B| \le 1-2\eps$, we have $|I| \ge 2\eps$. By construction, for $t \in I$, $g'(t)=-1/2$. 
Moreover, for $t \in I$, the compact set $C_t$ in $\T^d$, has size~$t$. Hence, $\eps \le |C_t| \le 1-\eps$,
$|A|+|B|+|C_t| \le 2-\eps$ and $2\max(|A|,|B|,|C_t|) \le |A|+|B|+|C_t|-\eps$. By hypothesis, for $t\in I$ we have
\[
 \int_0^t f(x)\,dx=\int_{C_t} \id_{A}*\id_{B}(x)\,dx
 \le -\alpha+\int_{C_t'} \id_{A'}*\id_{B'}(x)\,dx= -\alpha +\int_0^t g(x)\,dx.
\]
Using these inequalities together with the fact that $|I| \ge 2\eps$ and 
$g'(t)=-1/2$ for $t \in I$, \Cref{cont_arbitrary_functions} gives
\[
 \int_0^1 f(x)^n dx \le -\eps^{n-1}\alpha+\int_0^1 g(x)^n dx.
\]
As $f$ and $g$ are decreasing rearrangements of $\id_{A}*\id_{B}$ and $\id_{A'}*\id_{B'}$,
respectively, we deduce that $f^n$ and $g^n$ are decreasing rearrangements of
$(\id_{A}*\id_{B})^n$ and $(\id_{A'}*\id_{B'})^n$, respectively. Hence
\[
 \int_{\T^d} (\id_{A}*\id_{B}(x))^n dx= -\eps^{n-1}\alpha+ \int_{\T} (\id_{A'}*\id_{B'}(x))^n dx.
\]
The conclusion follows with $\tau=\eps^{n-1}$.
\end{proof}

\begin{proof}[Proof of \Cref{stab}]
Let $c=c(\eps/2)$ be the constant given by \Cref{Stab_Riesz-Sobolev} and let
$\tau=\tau(\eps/2, n)$ be the constant given by \Cref{Main_ingredient}.
Set $K=c\tau^{-1/2}(1-\eps)^{n/2}$. By approximating measurable sets with compact sets
and applying \Cref{lem_close}, we can assume $A$ and $B$ are compact sets in~$\T^d$.
By hypothesis, we know that
\[
 \E(|\medcup_i (A+b_i)|) \le \delta + \E(|\medcup_i (A'+b'_i)|),
\]
which after taking complements inside $\T^d$ and $\T$ is equivalent to
\[
 \E(|\medcap_i (A^c+b_i)|) \ge -\delta + \E(|\medcap_i (A^{\prime c}+b'_i)|).
\]
After scaling and using $\eps\le  |B|=|B'| \le 1-\eps$, we can rewrite the last inequality as 
\[
 \int_{b_1,\dots,b_n\in B} |\medcap_i (A^c+b_i)|\,db_1\dots db_n
 \ge -\delta (1-\eps)^n +
 \int_{b'_1,\dots,b'_n\in B'} |\medcap_i (A^{\prime c}+b'_i)|\,db'_1\dots db'_n,
\]
which is equivalent to
\[
 \int_{\T^d} |(x-A^c) \cap B|^n\,dx
 \ge -\delta (1-\eps)^n + \int_{\T}|(x-A^{\prime c}) \cap B'|^n\,dx,
\]
or in other words
\[
 \int_{\T^d}(\id_{A^c}*\id_{B}(x))^n dx
 \ge -\delta (1-\eps)^n + \int_{\T}(\id_{A^{\prime c}}*\id_{B'}(x))^n dx.
\]
To simplify the notation, set
\[
 D=A^c\qquad\text{and}\qquad
 D'=A^{\prime c} + \tfrac{1}{2}.
\]
Recall $B'$ is a centred interval with the same size as $B$ and note that $D'$
is a centred interval with the same size as~$D$. Moreover, note that $\eps \le |D|,|B| \le 1-\eps$. 
With this notation, we have
\begin{equation}\label{eq_equiv_cont}
 \int_{\T^d}(\id_{D}*\id_{B}(x))^n dx
 \ge -\delta (1-\eps)^n+\int_{\T}(\id_{D'}*\id_{B'}(x))^n dx.
\end{equation}

By \Cref{Main_ingredient} with parameters $\eps/2$, $n$ and
$\alpha=\tau^{-1}\delta(1-\eps)^n$ applied to the sets $D$ and~$B$, it follows that there exists
a compact set $C$ in $\T^d$ with $\eps/2 \le |C| \le 1-\eps/2$, $|D|+|B|+|C| \le 2-\eps/2$
and $2\max(|D|,|B|,|C|) \le |D|+|B|+|C|-\eps/2$ such that, if $C'$ is a centred interval
in $\T$ with the same size as~$C$, then
\[
 \int_{C}\id_{D}*\id_{B}(x)\,dx
 \ge - \tau^{-1}\delta (1-\eps)^n + \int_{C'}\id_{D'}*\id_{B'}(x)\,dx.
\]

By \Cref{Stab_Riesz-Sobolev} with parameters $\eps/2$
applied to the compact sets $D$, $B$ and $C$ in~$\T^d$, there exist parallel Bohr
sets $W$, $Y$ and $Z$ in $\T^d$ such that
\[
 |D\sdiff W|, |B\sdiff Y|, |C\sdiff Z|
 \le c\tau^{-1/2}\delta^{1/2}(1-\eps)^{n/2} = K\delta^{1/2}.
\]
Recalling that $D=A^c$, and setting $X=W^c$ (which is also a Bohr set, or more precisely the interior of a Bohr set, parallel
to $Y$ and~$Z$), we conclude that 
 \[
 |A\sdiff X|, |B\sdiff Y|, |C\sdiff Z| \le K\delta^{1/2}.
 \qedhere
\]
\end{proof}

\section{Sharp Stability Result  for \Cref{main_sphere}}\label{stability_sphere}

Given a measurable set $X$ in $\R^d$ we define $X'$ to be  the ball in $\R^d$ centred
around $0$ with the same volume as~$X$. 
In this section we prove the following sharp stability result for \Cref{main_sphere}. 

\begin{thm}\label{stab_Rd}
 For $\eps>0$ and\/ $d,n\in\N$, there exists $K=K(\eps,d,n)>0$ such that the following holds.
 Let\/ $A$ and\/ $B$ be measurable subsets of\/ $\R^d$ with $\eps \le |A|,|B| \le 1$.
 Let\/ $A'$ and\/ $B'$ be balls in $\R^d$ centred around $0$ with the same volumes
 as $A$ and\/ $B$ respectively. Assume that if we choose $b_1,\dots,b_n$ and\/ $b'_1,\dots,b'_n$
 uniformly at random from $B$ and\/~$B'$, respectively, then 
 \[
  \E(|A+\{b_1,\dots,b_n\}|)\le \delta+\E(|A'+\{b'_1,\dots,b'_n\}|).
 \]
 Then there exist homothetic ellipsoids $X$ and\/ $Y$ in $\R^d$ such that
 \[
  |A \sdiff X|,|B \sdiff Y|\le K\delta^{1/2}.
 \]
\end{thm}
Note that the exponent $1/2$ in \Cref{stab_Rd} is optimal by the same examples as
given after \Cref{stab}. We will make use of the following Riesz--Sobolev inequality,
due to Riesz~\cite{riesz}, together with its stability version, which
is due to Christ~\cite{Chr}.

\begin{thm}[\cite{riesz}]\label{Riesz-Sobolev_Rd}
 Let\/ $A$, $B$ and\/ $C$ be measurable subsets of\/~$\R^d$. Let\/ $A'$, $B'$ and\/ $C'$ be balls
 in $\R^d$ centred around\/ $0$ with the same volumes as $A$, $B$ and\/ $C$, respectively. Then
 \[
  \int_{C}\id_{A}*\id_{B}(x)\,dx \le \int_{C'}\id_{A'}*\id_{B'}(x)\,dx.
 \]
\end{thm}


\begin{thm}[\cite{Chr}, Theorem 2]\label{Stab_Riesz-Sobolev_Rd}
 For $\eps>0$ and\/ $d \in \N$ there exists $c=c(\eps,d)>0$ such that the following holds.
 Let\/ $A$, $B$ and\/ $C$ be compact subsets of\/ $\R^d$ with $\eps \le |A|^{1/d},|B|^{1/d},|C|^{1/d}\le 2$
 and\/ $2\max(|A|^{1/d},|B|^{1/d},|C|^{1/d}) \le |A|^{1/d}+|B|^{1/d}+|C|^{1/d}-\eps$.
 Let $A'$, $B'$ and\/ $C'$ be balls in $\R^d$ centred around\/ $0$ with the same volumes
 as $A$, $B$ and\/~$C$, respectively. Assume that
 \[
  \int_{C}\id_{A}*\id_{B}(x)\,dx \ge -\delta +\int_{C'}\id_{A'}*\id_{B'}(x)\,dx.
 \]
 Then there exist homothetic ellipsoids $X$, $Y$ and\/ $Z$ in $\R^d$ such that
 \[
  |A\sdiff X|, |B\sdiff Y|, |C\sdiff Z| \le c\delta^{1/2}.
 \]
\end{thm}

The main new ingredient for us is the following result.

\begin{thm}\label{Main_ingredient_Rd}
 For $\eps>0$ and\/ $d,n\in\N$ there exists $\tau=\tau(\eps,d,n)$ such that the following holds.
 Let\/ $A$ and\/ $B$ be compact sets in $\R^d$ with $2\eps \le |A|^{1/d},|B|^{1/d} \le 1$. Assume that
 \[
  \int_{C}\id_{A}*\id_{B}(x)\,dx \le -\alpha +\int_{C'}\id_{A'}*\id_{B'}(x)\,dx
 \]
 for every compact set\/ $C$ in $\R^d$ with $\eps \le |C|^{1/d} \le 2$ and
 \[
  2\max(|A|^{1/d},|B|^{1/d},|C|^{1/d}) \le |A|^{1/d}+|B|^{1/d}+|C|^{1/d}-\eps.
 \]
 Then 
 \[
  \int_{\R^d} 1-\bigg(1-\frac{\id_{A}*\id_{B}(x)}{|B|}\bigg)^n dx
  \ge \tau\alpha +\int_{\R^d} 1-\bigg(1-\frac{\id_{A'}*\id_{B'}(x)}{|B'|}\bigg)^n dx
 \]
\end{thm}

We shall derive this result from a general result about arbitrary functions. 

\begin{thm}\label{cont_arbitrary_functions_Rd}
 Let\/ $f,g \colon \R_+ \to [0,1]$ be continuous decreasing functions with bounded support with
 \[
  \int_0^\infty f(x)\,dx = \int_0^\infty g(x)\,dx
 \]
 and 
 \[
  \int_0^t f(x)\,dx \le \int_0^t g(x)\,dx
 \]
 for all $t\ge 0$. Assume that there exists an interval\/ $I \subseteq \R_+$
 with $g'(t)=-\gamma$ and
 \[
  \int_0^t f(x)\,dx \le -\beta+\int_0^t g(x)\,dx
 \]
 for all\/ $t\in I$. Then
 \[
  \int_0^\infty 1-(1-f(x))^n dx \ge (\gamma|I|)^{n-1}\beta+\int_0^\infty 1-(1-g(x))^n dx.
 \]
\end{thm}
\begin{proof}
Define $F(t)=\int_0^t f(x)dx$ and $G(t)=\int_0^t g(x)dx$ and assume $h\colon \R_+\to[0,\infty)$
is increasing and continuously differentiable. Then using integration by parts and choosing
$T$ sufficiently large so that $f(T)=g(T)=0$,
\begin{align*}
 \int_0^T (g(x)-f(x))h(x)dx
 &=(G(x)-F(x))h(x)\Big|_0^T-\int_0^T (G(x)-F(x))h'(x)dx\\
 &\le -\beta \int_I h'(x)=-\beta(h(b)-h(a)),
\end{align*}
where $I=[a,b]$ and we have used that 
$F(0)=G(0)$, $F(T)=G(T)$, $G(x)-F(x)\ge 0$ and $G(x)-F(x)\ge\beta$ on~$I$.

Apply this to the function $h(x)=(1-g(x))^{n-1}+(1-g(x))^{n-2}(1-f(x))+\dots+(1-f(x))^{n-1}$
and note that $h$ is increasing and $g(b)\le g(a)- \gamma(b-a) \le 1- \gamma(b-a) $, and so
\[
 h(b)-h(a)\ge (1-g(b))^{n-1}-(1-g(a))^{n-1}\ge (1-g(b))^{n-2}(g(a)-g(b))\ge (\gamma(b-a))^{n-1}.
\]
We deduce that
\begin{align*}
 \int_0^\infty 1-(1-g(x))^ndx-\int_0^\infty 1-(1-f(x))^ndx
 &=\int_0^\infty(g(x)-f(x))h(x)\,dx\\
 &\le -\beta(h(b)-h(a))\le - (\gamma |I|)^{n-1}\beta.
 \qedhere
\end{align*}
\end{proof}

\begin{proof}[Proof of \Cref{Main_ingredient_Rd}]
Given a continuous function $h \colon \R^d \to [0,1]$ with bounded support, as before, the
\emph{decreasing rearrangement} of $h$ is the unique
continuous decreasing function $\tilde h\colon [0,\infty) \to [0,1]$
with bounded support such that for every $\lambda \in [0,1]$ 
\[
 |\{x \colon h(x) \ge \lambda\}|=|\{x \colon \tilde h(x) \ge \lambda\}|.
\]

Note that $\id_{A}*\id_{B}/|B| \colon \R^d \to [0,1]$ and $\id_{A'}*\id_{B'}/|B'| \colon \R^d \to [0,1]$ are
continuous functions with bounded support. Let $f$ and $g$ be the decreasing rearrangements of
$\id_{A}*\id_{B}/|B|$ and $\id_{A'}*\id_{B'}/|B'|$, respectively. 

Let $v_d$ be the volume of the unit ball in $\R^d$; thus, $v_dr^d$ is the volume of the ball
of radius $r$ in $\R^d$. Note that $A'$ and $B'$ are balls centred at the origin of radii
$v_d^{-1/d}|A|^{1/d}$ and $v_d^{-1/d}|B|^{1/d}$, respectively.

Now $\id_{A'}*\id_{B'}/|B'| \colon \R^d \to [0,1]$ is spherically symmetrical
and $\id_{A'}*\id_{B'}(x)/|B'| \le \id_{A'}*\id_{B'}(y)/|B'|$ if $|x| \le |y|$.
Thus
\[
 g(v_d|x|^d)=\frac{\id_{A'}*\id_{B'}(x)}{|B'|}.
\]
Moreover, $g(x)=\min(|A|,|B|)/|B|$ for $x\le \rho_-^d$, where $\rho_-:=||A|^{1/d}-|B|^{1/d}|$, and $g(x)=0$ for
$x\ge \rho_+^d$, where $\rho_+:=|A|^{1/d}+|B|^{1/d}$, and $g$ has a continuous negative derivative
for $\rho_-^d<x<\rho_+^d$.

Set
\[
 I=[\rho_-+\eps, \rho_+-\eps]
\]
and
\[
 J=[(\rho_-+\eps)^d, (\rho_+-\eps)^d]
\]
Recall that $2\eps \le |A|^{1/d},|B|^{1/d} \le 1$ so that
$|I|\ge \eps$, $I \subseteq [\eps,2]$, $|J| \ge \eps^d$ and $J \subseteq [\eps^d, 2^d]$.
Also, there exists a $\gamma>0$ depending on $d$ and $\eps$ such that for $x \in J$
\[
 g'(x) \le -\gamma.
\]

For $t \in \R_+$ let $C'_t$ be the ball centred at the origin in $\R^d$ with volume~$t$.
Now as $\id_{A'}*\id_{B'}(x)$ is a decreasing function of~$|x|$, it follows that
\[
 \int_{C_t'}\frac{\id_{A'}*\id_{B'}(x)}{|B'|}\,dx = \int_0^t g(x)\,dx.
\]

Since $f$ is a decreasing rearrangement of $\id_{A}*\id_{B}/|B|$, it follows that for
$t\in\R_+$ there exists a compact set $C_t$ in $\R^d$ with $|C_t|=t$ such that 
\[
 \int_{C_t}\frac{\id_{A}*\id_{B}(x)}{|B|}\,dx = \int_0^t f(x)\,dx.
\]

By the definition of decreasing rearrangement, we have
\[
 \int_0^\infty f(x)\,dx
 =\int_{\R^d}\frac{\id_{A}*\id_{B}(x)}{|B|}\,dx
 =|A|=|A'|
 =\int_{\R^d}\frac{\id_{A'}*\id_{B'}(x)}{|B'|}\,dx
 =\int_0^\infty g(x)\,dx.
\]

For $t \in \R_+$, $C'_t$, $A'$ and $B'$ are centred balls with the same size as the
compact sets $C_t$, $A$ and $B$ in~$\R^d$, respectively, so \Cref{Riesz-Sobolev_Rd} implies 
\[
 \int_0^t f(x)\,dx
 =\int_{C_t}\frac{\id_{A}*\id_{B}(x)}{|B|}\,dx
 \le \int_{C'_t}\frac{\id_{A'}*\id_{B'}(x)}{|B'|}\,dx
 = \int_0^t g(x)\,dx.
\]

Moreover, for $t \in J$, the compact set $C_t$ in $\R^d$ has size $t$ so $|C_t|^{1/d} \in I$.
By the definition of~$I$, $\eps \le |C_t|^{1/d} \le 2$ and
$2\max(|A|^{1/d},|B|^{1/d},|C_t|^{1/d}) \le |A|^{1/d}+|B|^{1/d}+|C_t|^{1/d}-\eps$.
By hypothesis, for $t \in J$ we have
\[
 \int_0^t f(x)\,dx
 =\int_{C_t}\frac{\id_{A}*\id_{B}(x)}{|B|}\,dx
 \le -\frac{\alpha}{|B|}+\int_{C_t'}\frac{\id_{A'}*\id_{B'}(x)}{|B'|}\,dx
 = -\frac{\alpha}{|B|} + \int_0^t g(x)\,dx.
\]

Using the last three centred inequalities together with the fact that $|B| \le 1$, $|J| \ge \eps^d$
and $g'(t)\le -\gamma$ for $t \in J$, \Cref{cont_arbitrary_functions_Rd} gives
\[
 \int_0^\infty 1-(1-f(x))^n dx \ge (\gamma \eps^d)^n\alpha+\int_0^\infty 1-(1-g(x))^n dx.
\]

As $f$ and $g$ are decreasing rearrangements of $\id_{A}*\id_{B}/|B|$ and $\id_{A'}*\id_{B'}/|B'|$,
respectively, we deduce that $1-(1-f)^n$ and $1-(1-g)^n$ are decreasing rearrangements of
$1-\big(1-\frac{\id_{A}*\id_{B}(x)}{|B|}\big)^n$ and $1-\big(1-\frac{\id_{A'}*\id_{B'}(x)}{|B|}\big)^n$,
respectively. Hence
\[
 \int_{\R^d}1-\bigg(1-\frac{\id_{A}*\id_{B}(x)}{|B|}\bigg)^n dx
 \ge (\gamma \eps^d)^n\alpha +
 \int_{\R^d}1-\bigg(1-\frac{\id_{A'}*\id_{B'}(x)}{|B'|}\bigg)^n dx
\]
The conclusion follows with $\tau=(\gamma \eps^d)^n$.
\end{proof}

\begin{proof}[Proof of \Cref{stab_Rd}]
Let $c=c(\eps^{1/d}/2,d)$ be the constant given by \Cref{Stab_Riesz-Sobolev_Rd} and
$\tau=\tau(\eps^{1/d}/2,d,n)$ the constant given by \Cref{Main_ingredient_Rd}.
Set $K=c\tau^{-1/2}$. By approximating measurable sets with compact sets,
we can assume $A$ and $B$ are compact sets in~$\R^d$. From the hypothesis, we know that
\[
 \E(|\medcup_i (A+b_i)|) \le \delta + \E(|\medcup_i (A'+b'_i)|).
\]
This is equivalent to
\[
 \int_{\R^d}\Prb\big(x\in\medcup_i (A+b_i)\big)\,dx
 \le \delta + \int_{\R^d}\Prb\big(x\in\medcup_i (A'+b'_i)\big)\,dx,
\]
which can be expressed as
\[
 \int_{\R^d} 1-\bigg(1-\frac{|(x-A)\cap B|}{|B|}\bigg)^n dx
 \le \delta + \int_{\R^d} 1-\bigg(1-\frac{|(x-A')\cap B'|}{|B'|}\bigg)^n dx,
\]
or in other words
\[
 \int_{\R^d} 1-\bigg(1-\frac{\id_{A}*\id_{B}(x)}{|B|}\bigg)^n dx
 \le \delta + \int_{\R^d} 1-\bigg(1-\frac{\id_{A'}*\id_{B'}(x)}{|B'|}\bigg)^n dx.
\]

Recall that $A'$ and $B'$ are balls centred at the origin in $\R^d$ with the same size as the
compact sets $A$ and~$B$, respectively. Moreover, $\eps^{1/d} \le |A|^{1/d}, |B|^{1/d} \le 1$.

By \Cref{Main_ingredient_Rd} with parameters $\eps^{1/d}/2$, $d$, $n$ and
$\alpha = \tau^{-1}\delta$ applied to the sets $A$ and~$B$, it follows that there exists
a compact set $C$ in $\R^d$ with $\eps^{1/d}/2 \le |C|^{1/d} \le 2$ and
$2\max(|A|^{1/d},|B|^{1/d},|C|^{1/d}) \le |A|^{1/d}+|B|^{1/d}+|C|^{1/d}-\eps^{1/d}/2$ such that,
if $C'$ is a centred ball in $\R^d$ with the same size as~$C$, then
\[
 \int_C\id_{A}*\id_{B}(x)\,dx \ge -\tau^{-1}\delta + \int_{C'}\id_{A'}*\id_{B'}(x)\,dx.
\]

By \Cref{Stab_Riesz-Sobolev_Rd} with parameters $\eps^{1/d}/2$ and $d$ applied to the compact sets
$A$, $B$ and $C$ in $\R^d$, there exists homothetic ellipsoids $X$, $Y$, $Z$ in $\R^d$ such that
\[
 |A\sdiff X|, |B\sdiff Y|, |C\sdiff Z| \le c \tau^{-1/2}\delta^{1/2} = K \delta^{1/2}.
 \qedhere
\]
\end{proof}

\section{Open problems}

Perhaps the most interesting question that leads on from our results is to ask what would happen if we chose points uniformly at random not from the whole of $B$ but just from its
boundary (whether we are in Euclidean space or the torus). There are two different natural ways to quantify this
problem. For a given $\eps>0$, we could choose uniformly from the set $(1+\eps)B - B$, or we could
choose uniformly from $B_\eps - B$, where $B_\eps$ denotes the set of all points at distance at most
$\eps$ from~$B$. The difference between these two lies in how they `bias' the boundary measure towards the actual
shape and curvature of the boundary -- we feel that both forms are very natural for this problem.

In a different direction, what happens if we use $n$ random points in $A$ and in $B$ (in $\R^d$) and take the convex hull?
Is it still true that Euclidean balls minimize the expected value of this quantity? Here $A$ and $B$
would be convex sets of given volumes.

In the results in the paper, we have formed the sumset of $A$ with $n$ random points of~$B$. But what would
happen if, instead of \emph{points}, we used random copies of a fixed body? For example, what would happen
if we used unit intervals, with centres chosen uniformly at random from $B$ but with orientations also chosen at random? To be more precise, given our sets
$A$ and $B$ in $\R^d$, we first
choose $n$ points uniformly at random from $B$ and then for each of these points we take a unit interval centred at
this point with direction chosen uniformly at random -- 
we wish to minimize the expected volume of the sum of
$A$ with the union of these $n$ intervals.

One could also allow this `extra' body (above, the unit
interval) to vary not by a random rotation but rather by being related to~$B$:
perhaps the most natural example would be if it were a homothetic copy of~$B$. Thus one could
ask for example the following question: given convex sets $A$ and $B$ in $\R^d$ of given volumes, how do we minimize
the expected volume of the sumset of $A$ with $n$ random copies of $B/2$, chosen uniformly at random from
the copies of $B/2$ inside~$B$? In an equivalent, but
perhaps less attractive, formulation, this is asking to
minimize the sum of $A+B$ with $n$ random points of $B$.


\begin{thebibliography}{10}

\bibitem{Bar}
Barthe,~F.,
Autour de l'in\'egalit\'e de Brunn--Minkowski,
\emph{Ann. Fac. Sci. Toulouse Math.} \textbf{12} (2003) 127--178.

\bibitem{Chr}
Christ, M.,
A sharpened Riesz-Sobolev inequality,
\texttt{arXiv:1706.02007} (2017).

\bibitem{Christ}
Christ, M. and M. Iliopoulou, 
Inequalities of Riesz--Sobolev type for compact connected abelian groups,
\emph{Amer. J. Math.} \textbf{144} (5) (2022) 1367--1435.

\bibitem{Gard}
Gardner,~R.J.,
The Brunn--Minkowski inequality,
\emph{Bull. Amer. Math. Soc.} \textbf{39} (2002) 355--405.

\bibitem{Knes}
Kneser, M.,
Summenmengen in lokalkompakten abelschen Gruppen,
\emph{Math. Z.} \textbf{66} (1956) 88--110.

\bibitem{macbeath}
Macbeath, A.M., On measure of sum sets. II: The sum-theorem for the torus,
\emph{Proc. Cam. Phil. Soc.} \textbf{49} (1953) 40--43.

\bibitem{pollard}
Nazarewicz E, M. O'Brien, M. O'Neill and C. Staples, 
Equality in Pollard’s theorem on set addition of congruence classes, 
\emph{Acta Arithmetica} \textbf{127} (2007) 1--15.

\bibitem{riesz} Riesz, F., Sur une in{\'e}galit{\'e} int{\'e}grale, 
\emph{J. London Math. Soc.} \textbf{5}
(1930) 162--168.

\bibitem{Tao}
Tao, T., 
An inverse theorem for an inequality of Kneser,
\emph{Proc. Steklov Institute Math.} \textbf{303} (2018) 193--219.

\bibitem{volcic} Vol\v{c}i\v{c}, A., 
Random Steiner symmetrizations of sets and functions, 
\emph{Calculus of Variations and Partial Differential Equations} \textbf{46} (2013) 555--569.

\end{thebibliography}
\end{document}